\documentclass[11pt]{amsproc}
\usepackage[british]{babel}

\usepackage{a4wide}
\usepackage{setspace}
\usepackage{amssymb}
\usepackage{amsmath}
\usepackage{amsthm}
\usepackage{array}
\usepackage{enumitem}
\usepackage{amsrefs}
\usepackage[citecolor=blue]{hyperref}

\usepackage[usenames]{color}%%  NEW
\newtheorem{thm}{Theorem}[section]
\newtheorem{lem}[thm]{Lemma}

\newtheorem{prop}[thm]{Proposition}

\theoremstyle{remark}
\newtheorem{rem}[thm]{Remark}

\def\rZ{{\mathrm Z}}

\def\NN{{\mathbb N}}

\def\GL{{\mathrm {GL}}}
\def\rk{{\mathrm {rk}}}
\def\id{{\mathrm {id}}}

\def\ep{\varepsilon}
\def\leq{\leqslant}
\def\geq{\geqslant}

\def\SL{{\mathrm {SL}}}
\def\PSU{\mathrm{PSU}}

\def\PSL{\mathrm{PSL}}

\def\Sp{{\mathrm {Sp}}}

\def\SU{{\mathrm {SU}}}

\def\CU{{\mathcal U}}

\def\phi{\varphi}

\title[Metric ultraproducts of finite simple groups]{Some geometric properties of metric ultraproducts \\ of finite simple groups}

\author{Andreas Thom}
\address{A.T., Institut f\"ur Geometrie, TU Dresden, Germany}
\email{andreas.thom@tu-dresden.de}

\author{John Wilson}
\address{J.W., Mathematical Institute, Oxford, United Kingdom}
\email{john.wilson@maths.ox.ac.uk}

\begin{document}

\onehalfspace

\begin{abstract}
In this article we prove some previously announced results about metric ultrapro\-ducts of finite simple groups. We show that any non-discrete metric ultraproduct of alternating or special linear groups is a geodesic metric space. For more general non-discrete metric ultraproducts of finite simple groups, we are able to establish path-connectedness. As expected, these global properties reflect asymptotic properties of various families of finite simple groups.
\end{abstract}

\maketitle

\medskip

\section{Introduction}

The study of metric ultraproducts of finite groups that are equipped with natural geometrically or combinatorially defined bi-invariant metrics has attracted some interest recently. First of all, Elek and Szab\'o \cite{elek} established that a discrete group is sofic if and only it can be embedded in a metric ultraproduct of alternating groups. Secondly, the model theory of metric structures (as developed in \cite{MR2436146}) provides a natural setting in which to study asymptotic metric properties of families of metric spaces. Families of finite simple groups serve as interesting examples in this context.
Some results about the metric geometry of such ultraproducts were announced in \cite{MR3210125}. The aim of this article is to give complete proofs of some of these results. We refer for the basic terminology to \cite{MR3210125} and \cite{MR3162821}.

Throughout the paper, $\omega$ is a fixed non-principal ultrafilter on $\NN$.  Let $(G_i,d_i)_{i \in \NN}$ be a sequence of groups with normalized bi-invariant metrics. Then the metric ultraproduct $(G,d):=\prod_{i \to \omega} (G_i,d_i)$ is again a metric space, with metric induced by the pseudometric 
$$d((g_i),(h_i)):= \lim_{i \to \omega} d_i(g_i,h_i)$$
on the abstract ultraproduct $\prod_{i\to\omega}G_i$. The metric ultraproduct is the quotient of the abstract ultraproduct that makes this pseudo-metric a metric; see \cite{MR3162821}.  
It is clear that the metric geometry of $G$ reflects asymptotic features of the behaviour of the metrics $d_i$.
When the groups $G_i$ are finite simple groups, the most natural choices of bi-invariant metrics are asymptotically equivalent \cite{MR3162821} and so the topology on $G$, unlike the metric itself, is independent of the choices. In our context, the alternating groups $A_n$ come equipped with the normalized Hamming distance
$$d(\sigma,\tau) = \frac{\mu_{\sigma^{-1}\tau}}n, \quad \hbox{for all }\sigma,\tau \in A_n,$$
where for each permutation $\sigma$ we write $\mu_\sigma \in \{0,\dots,n\}$ for the number of elements moved by $\sigma$.
The groups of Lie type are usually considered with the normalized rank metric or projective rank metric, coming from the natural embedding into either ${\rm GL}_n(F)$ or ${\rm PGL}_n(F)$ for a field $F$ and integer $n$; the normalized rank metric on ${\rm GL}_n(F)$ is defined by
$$d(x,y) = \frac{\rk(x-y)}n, \quad \hbox{for all }x,y \in {\rm GL}_n(F),$$
and the projective rank metric on ${\rm PGL}_n(F)$ by
$$d(\bar x,\bar y) = \inf_{\lambda \in F^{\times}}\frac{\rk(x-\lambda y)}n, \quad \hbox{for all }x,y \in {\rm PGL}_n(F),$$
where $x,y \in {\rm GL}_n(F)$ denote arbitrary lifts of $\bar x,\bar y \in {\rm PGL}_n(F)$. All of these metrics are equivalent to the more intrinsically defined conjugacy metrics. See \cite{MR3162821} for details.

The following result, which is essentially known, clarifies the implications of the classification of finite simple groups for the structure of metric ultraproducts of finite simple groups.

\begin{thm}\label{connected}  
Let $(G_i)_{i \in \NN}$ be a sequence of $($non-abelian$)$ finite simple groups with bi-invariant metrics and let $G$ be the metric ultraproduct $\prod_{i\to\omega} (G_i,d_i)$.    
If $G$ is infinite then one of the following holds$:$ 
\begin{enumerate}
\item[\rm(i)] $G$ is a Chevalley group $($possibly twisted\/$)$ over an ultraproduct of finite fields  
and the metric on $G$ is discrete$;$ 
\item[\rm(ii)] $G$ is a metric ultraproduct of alternating groups %with Hamming metric 
of ranks $n_i$ with $n_i\to_\omega\infty$, or
\item[\rm(iii)] a metric ultraproduct of finite simple classical groups %, with projective rank as metric, 
all of the same type
$($linear, symplectic, unitary or orthogonal\/$)$ over finite fields and of ranks $n_i$ with $n_i\to_\omega\infty$.
\end{enumerate}
\end{thm}

\begin{proof}
With finitely many exceptions  each (non-abelian) finite simple groups is isomorphic to an alternating group or a (possibly twisted) group of Lie type ${}^j{\mathrm X}_n(K)$ for a suitable finite field $K$, 
where ${}^j{\mathrm X}_n$ belongs to the list$$
{\mathrm  A}_n, {\mathrm  B}_n, 
{\mathrm C}_n, {\mathrm D}_n,  
{}^2{\mathrm  A}_n,  
 {}^2{\mathrm  D}_n\eqno(\ast)$$ 
 or to the list $${\mathrm E}_6, {\mathrm  E}_7, {\mathrm E}_8,
{\mathrm F}_4, {\mathrm G}_2,  
{}^2{\mathrm B}_2, {}^3{\mathrm D}_4,  {}^2{\mathrm G}_2, {}^2{\mathrm  E}_6,
{}^2{\mathrm F}_4.$$ 
Since an ultrafilter always selects one from a finite number of possibilities, it follows that $G$ is isomorphic to one of the following: (a) an ultraproduct of isomorphic simple groups, (b) a group as in (ii) in the statement of the theorem, 
(c) an ultraproduct of groups 
${}^jX_n$ with $X, j$ fixed.  The groups in (a) are finite, and if in (c) the integers $n$ are unbounded on each set in $\CU$ then $G$ is as in (iii).  If in case (c) the integers $n$ are bounded on a set in $\CU$, then they can be assumed constant, 
and then a theorem of F.\ Point \cite{point} asserts that $G$ is itself a group ${}^jX_n(F)$ for some field $F$;  
being simple it coincides with the metric ultraproduct.  From \cite[Corollary 1.9]{LS} there is a constant $c$ such that each $G_i$ has the property that for all non-trivial $g_1,g_2\in G_i$ the element $g_2$ is a product of at most $c$ conjugates of $g_1$.  So therefore does $G$. Choose $h\in G\setminus\{1\}$ and let $d(h,1)=\kappa$.  If $g\in G\setminus\{1\}$ then since $h$ is a product of at most $c$ conjugates of $g$ we have $\kappa\leq cd(1,g)$.  Therefore $d(1,g)\geq\kappa/c$ and the result follows.
\end{proof}

\begin{rem} If in (ii), (iii) above the group $G$ is simple then it follows from 
a result of \cite{nik} that the topology on $G$ is induced by the conjugacy metrics on the groups $G_i$, and then from results of \cite{MR3162821} that the topology is induced from the Hamming metrics in case (ii) and from the projective rank metrics in case (iii).
\end{rem}

It follows from more general facts that the non-discrete simple groups in Theorem \ref{connected} are complete; see \cite[p.\ 147]{wantiez}.  Here we shall show that they are path-connected.  We shall deduce this as a consequence of a stronger result.
We recall that a {\em geodesic} from one point $x$ of a metric space $(X,d)$ to another point $y$ is an isometric embedding $p \colon [0,d(x,y)] \to X$ such that $p(0)=x$ and $p(d(x,y))=y$, and that $X$ is called a {\em geodesic space} if all pairs of points are connected by geodesics. A space $X$ is called {\em star-shaped} if it has a point $x_0$ that is connected to each element of $X$ by a geodesic.  

Our main result is the following theorem.

\begin{thm}\label{main}
The non-discrete metric ultraproducts of sequences of alternating groups equipped with the normalized Hamming metrics or projective linear groups equipped with the projective rank metrics are complete geodesic metric spaces.

Moreover, every non-discrete simple metric ultraproduct $G$ of finite simple groups has geodesic subgroups $H_1,H_2,H_3$ and a star-shaped subspace $H_4$ such that $G=H_1H_2H_1H_3H_4$.
\end{thm}

If follows that the groups in the first part above are path-connected.  
Let $G$ and the subspaces $H_i$ be as in the second part.  Given $x,y\in G$ we write
$x^{-1}y=h_1\dots h_5$ with factors $h_j$ in the appropriate subsets $H_i$.  Concatenating translates of geodesics from $1$ to the elements $h_j$ we obtain a path from $x$ to $y$. In particular, all non-discrete metric ultraproducts of  finite simple groups are path-connected.

\section{Metric ultraproducts of groups $A_n$ and $\PSL_n$} 
To prove that certain groups are geodesic metric spaces we use the following known result. For the reader's convenience we include a proof.

\begin{prop}\label{spaces} Let $(X_i,d_i)$ be a sequence of metric spaces, and $(\ep_i)$ a sequence of positive reals with $\ep_i\to_\omega0$.  
Suppose that for each $i$, and for all $x_i$, $y_i$ in $X_i$ there exists $z_i\in X_i$ with $d_i(x_i,z_i)\leq\frac12d_i(x_i,y_i)+\ep_i$ and $d_i(z_i,y_i)\leq\frac12d_i(x_i,y_i)+\ep_i$.  Then the metric ultraproduct $X=\prod_{i \to \omega}(X_i,d_i)$ is a geodesic space. \end{prop}
\begin{proof}  
Let $x,y$ be elements of $X$, represented by sequences $(x_i)_{i \in \NN}$ and $(y_i)_{i \in \NN}$. %\achtung{(Without loss of generality, we may assume that $d(x,y)=1$ How? I don't see why and we don't need it.)} 
Let $d(x,y)=\ell$ and let $D$  be the set of dyadic rationals in the closed unit interval $[0,1]$.  We define $p(\lambda\ell)$ for all $\lambda\in D$ recursively as follows.  
 Set $p_i(0)=x_i$ and $p_i(\ell)=y_i$.  If $\lambda$ is of form $2^{-k}m$ with
 $0<k\in\NN$ and $m$ an odd integer, define $p_i(\lambda\ell)$ to be a point $z_i\in X_i$ satisfying
 $$d_i(p_i((\lambda+j 2^{-k})\ell), z_i)\leq \frac12
d_i(p_i((\lambda -2^{-k})\ell), p_i((\lambda +2^{-k})\ell)+\ep_i\quad\hbox{for } j=\pm1.$$

An easy induction on $k+k'$ shows that if $\lambda=2^{-k}m$ and $\lambda=2^{-k'}m'$ with $k,k'\geq0$ and $m,m'$ odd then 
$$\left|d_i(p_i(\lambda),p_i(\lambda')) - |\lambda-\lambda'| \right| \leq (k+k') \varepsilon_i.$$
Therefore the sequence $(p_i)_{i \in \mathbb N}$ induces an isometry $p$ from $\{\lambda\ell\mid \lambda\in D\}$ to
$X$ with $p(0)=x$ and $p(\ell)=y$.
Since the metric ultraproduct is complete, the map $p$ admits a unique extension to an isometric embedding from $[0,\ell]$ to $X$, and this shows that the metric ultraproduct is a geodesic space.
\end{proof}

\begin{lem}\label{splitAn}   Let $t\in[0,1]$ and let
$g$ be an even permutation.
% with support of size $n$. 
Then there are even permutations $h, k$ with $g=hk$ such that  $|\mu_h-t\mu_g|\leq\frac32$ and $|\mu_k-(1-t)\mu_g|\leq\frac32$.
\end{lem}

\begin{proof} %The bound $\frac32$ is achieved when $g$ is a $3$-cycle and $t=\frac12$.
It is well known that for each cycle $\pi$ and integer
$l$ with $1<l<\mu_\pi$ we can write
$\pi=\lambda\rho$ with $\lambda$, $\rho$ of lengths $l$, $\mu_\pi-l+1$.   

Consider the disjoint cycle decomposition of $g$; let $h_1$ be a product
of as many of the cycles as possible subject to being even and satisfying $\mu_{h_1}\leq t\mu_g$, and
$k_1$  a product of as many cycles of $h_1^{-1}g$ as possible with $k_1$ even and $\mu_{k_1}\leq (1-t)\mu_g$.  We may suppose that $g\neq h_1k_1$ since otherwise the conclusion is clear. Set $r_1=
t\mu_g-\mu_{h_1}$ and $r_2=(1-t)\mu_g-\mu_{k_1}$.
If $d$ is a cycle of odd length or a product of two cycles of even length in the cycle decomposition of $k_1^{-1}h_1^{-1}g$ then $\mu_d> \max(r_1,r_2)$. Hence $k_1^{-1}h_1^{-1}g=d$ and $\mu_d=\mu_g-\mu_{h_1}-\mu_{k_1}=r_1+r_2$.  

First suppose that $d$ is a cycle.  If $r_1\leq \frac32$ take $h=h_1$; if $r_1>\frac32$ and $r_2\leq \frac32$ take $h=h_1d$.  Otherwise $\mu_d\geq5$.  Let $l$ be the odd integer satisfying $r_1-\frac12\leq l<r_1+\frac32$.  Then $3\leq l\leq \mu_d-2$
and $r_2-\frac12< \mu_d-l+1\leq r_2+\frac32$.
Write $d$ as a product of cycles $h_2$, $k_2$ of lengths $l$ and $\mu_d-l+1$. It follows easily that the even permutations $h=h_1h_2$ and $k=k_2k_1$ have the required property. 

The procedure if $d$ is a product of two cycles of even length is similar; one of the cycles of $d$ has length $l\leq\max(t\mu_g-\mu_{h_1},(1-t)\mu_g-\mu_{k_1})$ and we take it as one of the cycles of $h_2$ or $k_2$.  The other cycle of length $\mu_d-l$ can be written as a product of two cycles one of which has odd length in any prescribed interval of length $2$ in $[0,\mu_d-l]$, and the result follows. 
\end{proof}

\begin{lem}\label{splitSLn} 
Let $g$ be an element of $\GL_n(F)$ where $F$ is a field. Suppose that $g-\lambda 1$
has rank $r$ where $\lambda\neq0$, and let $\phi_1,\phi_2$ be non-negative real numbers with sum $r$.  
Then \begin{enumerate} \item[\rm(a)] there exist elements $h,k\in\GL_n(F)$ with $g=hk$ such that
$$|\rk(h-\lambda1)-\phi_1|\leq2\quad\hbox{and}\quad |\rk(k-1)-\phi_2|\leq2.\eqno(\ast)$$
\item[\rm(b)]If in addition $g\in\SL_n(F)$ then there exist $h,k\in\SL_n(F)$ with $g=hk$ such that
$$|\rk(h-\lambda1)-\phi_1|\leq3\quad\hbox{and}\quad |\rk(k-1)-\phi_2|\leq3.$$
\end{enumerate}\end{lem}

\begin{proof} (a) The element $g_1=\lambda^{-1}g$ satisfies the hypothesis with $\lambda=1$.  If $g_1=h_1k$ and  
$h_1-1$ and $k-1$ have ranks satisfying the equivalents for $g_1,h_1,k$ of the inequalities in $(\ast)$ then $h:=\lambda h_1$ and $k$ have the correct property.  Therefore we can suppose that $\lambda=1$.
We argue by induction on $n$.  Since result is clear for $n=1$ we assume that $n\geq2$.

First suppose that the natural module $V$ for $\GL_n(F)$ is a cyclic (right) $F\langle g\rangle$-module.
Let $e_1$ be a module generator and define $e_{i+1}=e_1(g-1)^i$ for $i=1,\dots,n-1$. Thus $e_1,\dots, e_n$ is a basis for $V$,
and $r=\rk(g-1)$ is $n$ or $n-1$.  Write $s=\lceil \phi_1\rceil$.  If $s=n$ we take $h=g$, $k=1$; if $s=0$ we take $h=1, k=g$.
Otherwise,  let $h$ be the invertible map
defined by $e_i\mapsto e_i+e_{i+1}$ for $i\leq  s$, $e_{s+1}\mapsto e_1$ and $e_i\mapsto e_i$ for $i>s+1$.  So $\rk(h-1)=s+1$.  Moreover $k:=h^{-1}g$ fixes $e_2,\dots, e_{s+1}$, and so $\rk(k-1)\leq n-s\leq n-\phi_1\leq 1+\phi_2$.   Finally, since $g-1=h(k-1)+h-1$ we have
$$\rk(k-1)\geq\rk(g-1)-\rk(h-1)\geq r-s-1\geq r-\phi_1-2=\phi_2-2.$$
The inequalities in $(\ast)$ follow. 

Since in general $V$ is a direct sum of cyclic $F\langle g\rangle$-modules, we can now assume that
$V$ is a direct sum  $U\oplus W$ of non-zero $F\langle g\rangle$-submodules.  
First suppose that $r_1=\rk(g-1)|_U$ satisfies $r_1\leq\phi_1$.
Then $\rk(g-1)|_W= r-r_1$,
and by induction applied to $W$ and the numbers $\phi_1-r_1$, $\phi_2$ we can find
$h_2,k_2\in\GL(W)$ with $g|_W=h_2k_2$ and 
$$|\rk(h_2-1)-(\phi_1-r_1)|\leq1\quad\hbox{and}\quad |\rk(k_2-1)-\phi_2|\leq1.$$
Define $h,k$ to be the maps acting as $h_2$, $k_2$ on $W$ and acting respectively as $g$ and the identity on $U$. It is clear that $h,k$ have the required properties.

Similarly the result follows if $r_2=\rk(g-1)|_W$ satisfies $r_2\leq \phi_2$.  Since $r_1+r_2=r=\phi_1+\phi_2$ one of the inequalities $r_1\leq\phi_1$, $r_2\leq \phi_2$ must
hold, and (a) follows.

(b) Let $g\in\SL_n(F)$.  By (a) we can write $g=h'k'$ with 
$$|\rk(h'-\lambda1)-\phi_1|\leq2\quad\hbox{and}\quad |\rk(k'-1)-\phi_2|\leq2.$$
Let $d$ be the diagonal matrix with first diagonal entry $\det k'$ and remaining diagonal entries $1$,
and define $h=h'd$, $k=d^{-1}k'$.  Then $h,k\in\SL_n(F)$ and $g=hk$.  Since
$h-\lambda1=(h'-\lambda1)d+\lambda(d-1)$ we have
$|\rk(h-\lambda1)-\rk(h'-\lambda1) |\leq 1$.  Similarly $|\rk(k-1)-\rk(k'-1) |\leq 1$.   The assertion now follows. 
\end{proof} 

\begin{proof}[Proof of Theorem {\rm \ref{main}}]
First let $G=\prod_{i \to \omega}(A_{n_i},d_i)$ be a metric ultraproduct of groups $A_{n_i}$ with the Hamming metric, and with $n_i\to_\omega\infty$.  Define $\ep_i=3/{2n_i}$ for each $i$; thus $\ep_i\to_\omega0$.  Given $x_i,y_i\in A_{n_i}$, set $g=x_i^{-1}y_i$ and in the notation of Lemma \ref{splitAn} find $h,k\in A_{n_i}$ with $g=hk$, and
$$\mu_h\leq\textstyle\frac12 \mu_g+\frac32,\quad \mu_k\leq\frac12 \mu_g+\frac32.$$
Thus $d_i(1,h)\leq\frac12d_i(1,g)+\ep_i$, $d_i(1,k)\leq\frac12d_i(1,g)+\ep_i$.
Hence $d_i(x_i,x_ih)\leq\frac12d_i(x_i,y_i)+\ep_i$, and
$$\textstyle d_i(x_ih,y_i)=d_i(x_ig,y_ik)=d_i(y_i,y_ik)=d_i(1,k)\leq \frac12d_i(x_i,y_i)+\ep_i.$$  Therefore the assertion of (a) for alternating groups follows from Proposition \ref{spaces}.

The argument for ultraproducts $G=\prod_{i \to \omega} (\PSL_{n_i}(F_i),d_i)$ of (not necessarily finite) groups of type $\PSL$ with the rank metric is similar.  Define $\ep_i=3/n_i$.  Let $\bar x_i,\bar y_i\in\PSL_{n_i}(F_i)$ and let $x_i,y_i$ be preimages in $\SL_{n_i}(F_i)$.  Set $g=x_i^{-1}y_i$, choose $\lambda\in F\setminus\{0\}$ with $r:=\rk(g-\lambda1)$ as small as possible, and using Lemma \ref{splitSLn} find $h,k\in \SL_{n_i}(F_i)$ with $\rk(h-\lambda1)\leq\frac12r+3$, $\rk(k-1)\leq\frac12r+3$.  Let $\bar h_i$ be the image of $h$ in $\PSL_{n_i}(F_i)$. Then $d_i(\bar x_i,\bar y_i)=r$ and
$$\textstyle d_i(\bar x_i,\bar x_i\bar h_i)\leq\frac12r+\ep_i,\quad  d_i(\bar x_i\bar h,\bar y_i)\leq\frac12r+\ep_i.$$ Hence the result follows from Proposition \ref{spaces}.  \end{proof}%\hfill $\Box$

\section{Geodesic and star-shaped subspaces}
Our treatment of metric ultraproducts of the other classical groups depends basically on the fact that these groups have a BN-pair.  
As preparation, we identify some geodesic subgroups of metric ultraproducts $\prod_{i \to \omega} (\PSL_{n_i}(F_i),d_i)$.  For the rest of the paper (with a brief exception in the proof of Lemma \ref{unipotentpath} (b)), all metrics are defined in terms of projective rank and they are denoted simply by $d$.  For the classical groups under consideration, we do not require the fields to be finite. We begin with a technical calculation.  

In the next two lemmas, for $n\geq2$ and an element $\lambda$ of a field $F$, we write $m_{\lambda,n}$ for the diagonal matrix with first diagonal entry $\lambda^{-(n-1)}$ 
and remaining diagonal entries $\lambda$.

\begin{lem}\label{ultimatepatch} Let $n\geq2$, and $F$ be a field. Let $E$ be the subgroup $\{m_{\lambda,n}\mid\lambda\neq0\}$ of $\SL_n(F)$. Let $X$ be the set of matrices $a\in\GL_n(F)$ such that
$$\dim\ker(a-\lambda 1_n)\leq \dim\ker (a-1_n)+2\quad\hbox{for all }\lambda\neq0.\eqno(\ast\!\ast\!\ast)$$
Then $\GL_n(F)=EX$.
\end{lem}

\begin{proof} Let $a\in\GL_n(F)$, and choose $\lambda\neq0$ with $s=\dim\ker (a-\lambda1_n)$ as large as possible. Set $x=m_{\lambda^{-1},n}a$.  Since  $(m_{\lambda^{-1},n}-\lambda^{-1}1_n)a$ has rank at most $1$, we have
$$\dim\ker(x-1_n)=\dim\ker((m_{\lambda^{-1},n}-\lambda^{-1}1_n)a+\lambda^{-1}(a-\lambda1_n))\geq s-1,$$
while for each $\mu\neq0$ we have
$$\dim\ker(x-\mu1_n)=\dim\ker((m_{\lambda^{-1},n}-\lambda^{-1}1_n)a+\lambda^{-1}(a-\lambda\mu1_n))\leq s+1.$$  Thus $x\in X$ and the result follows.  \end{proof}

\begin{lem}\label{unipotentpath}  Let $G=\prod_{i \to \omega} (\PSL_{2n_i}(F_i),d)$, where $n_i\to_\omega\infty$.
\begin{enumerate} \item[\rm(a)] Let $U_i$, $V_i$ be respectively the subgroup of $\SL_{2n_i}(F_i)$ of all block matrices $\big(\begin{smallmatrix} 1& k\cr 0&1 \end{smallmatrix}\big)$ with $k$ skew-symmetric with zero diagonal and the subgroup of all matrices $\big(\begin{smallmatrix} 1& k\cr 0&1 \end{smallmatrix}\big)$ with $k$ symmetric, and let $U_i^T,V_i^T$ be the images of $U_i$, $V_i$ under the transposition map $T$.  Then the images in $G$ of $\prod_{i \to \omega} (U_i,d)$, $\prod_{i \to \omega}(U_i^T,d), \prod_{i \to \omega}(V_i,d)$ and $\prod_{i \to \omega}(V_i^T,d)$ are geodesic subgroups.
\item[\rm(b)]   Suppose that \ $\bar{ }$ \ is an automorphism of each $F_i$ of order at most $2;$ write \ $\bar{ }$ also for the map induced on $\GL_{n_i}(F_i)$.  For each $i$ set
$$H_i=\left\{ \begin{pmatrix} a^{-1}& 0\cr 0& \bar a^T\end{pmatrix} \mid a\in\GL_{n_i}(F_i)\right\},\quad E_i=\left\{\begin{pmatrix}m_{\lambda^{-1},n_i}&0\cr0&m_{\bar\lambda,n_i}\end{pmatrix}\mid\lambda\neq0\right\}
$$ 
and
let $Y_i$ be the set of elements $\big(\begin{smallmatrix} a^{-1}& 0\cr 0& \bar a^T\end{smallmatrix}\big)$ in $H_i$ with $a$ satisfying the condition $(\ast\!\ast\!\ast)$.
Then the images $H$, $E$, $Y$ of $\prod_{i \to \omega}(H_i\cap\SL_{2n_i}(F_i),d)$, $\prod_{i \to \omega}(E_i,d)$, $\prod_{i \to \omega}(Y_i\cap\SL_{2n_i}(F_i),d)$   
 satisfy $H=EY;$ the elements in $E$ can be joined to the identity by a path of length $1$ and $Y$ a star-shaped space.
\end{enumerate} 
 \end{lem}

\begin{proof}  (a) Let $U$ be the image in $G$ of $\prod_{i \to \omega} (U_i,d)$ and let $u\in U$ be represented by the sequence $(u_i)$ with $u_i\in U_i$ for each $i$.
Let $u_i=\big(\begin{smallmatrix} 1& k_i\cr 0&1 \end{smallmatrix}\big)$.  Then there is a matrix $c_i\in\GL_{n_i}(F_i)$ such that $k_i=c_id_ic_i^T$, where $d_i$ is the block diagonal matrix with $s_i$  
diagonal $2\times2$ entries $\big(\begin{smallmatrix} 0& 1\cr -1&0 \end{smallmatrix}\big)$ and $n_i-2s_i$ scalar entries $0$; so $\rk(u_i-1)=2s_i$, see \cite{albert} for this and similar facts especially in characteristic 2. For $t\in[0,1]$ let $d_i(t)$ be the matrix having the same
first $2\lfloor s_it\rfloor$ rows as $d_i$ and remaining rows zero, and let $k_i(t)=c_id_i(t)c_i^T$.  Then 
$$\rk\left(\begin{pmatrix}1&k_i(t)\cr 0 &1\end{pmatrix}-1\right)=\textstyle 2\lfloor s_it\rfloor,$$ which differs from
$t\,\rk(u_i-1)$ by at most $2$.   On the other hand, for $\lambda\neq0,1$ we have
$$\rk\left(\begin{pmatrix}1&k_i(t)\cr 0 &1\end{pmatrix}-\lambda1\right)=2n.$$
So the composite of the function $t\mapsto\big(\big(\begin{smallmatrix} 1& k_i(t)\cr 0 & 1 \end{smallmatrix}\big)\big)$ and the quotient map from the product of the groups $U_i$ to $U$ is a geodesic in $U$ from $1$ to $u$. 

The proof for the image of $\prod_{i \to \omega}(V_i,d)$ is similar but slightly more complicated.  For $k_i$ symmetric we can find $c_i\in\GL_{n_i}(F_i)$ with $k_i=c_id_ic_i^T$ and with $d_i$ a block diagonal matrix with $s_i$ entries $\big(\begin{smallmatrix} 0& 1\cr 1&0 \end{smallmatrix}\big)$, then $s_i'$ non-zero diagonal entries, and finally $n_i-2s_i-s_i'$ zero entries; here, $s_i,s_i'$ satisfy $2s_i+s_i'=\rk(k_i)$.  We take $d_i(t)$ to be the matrix having the same first $2\lfloor(s_i+\frac12 s_i')t\rfloor$ rows as $d_i$ and remaining rows zero, and argue as in the above paragraph.  The same arguments give the result for the transposed subgroups.

(b) It follows easily from Lemma \ref{ultimatepatch} 
that $H_i\cap\SL_{2n_i}(F_i)=E_i(Y_i\cap\SL_{2n_i}(F_i))$ for each $i$ and hence $H=EY$.  An argument very similar to the proof in (a) shows that every element of $E$ can be joined to the identity element by a path in $H$ of length at most $1$.
We now prove that for each $y\in Y$ there is a geodesic in $Y$ connecting $1$ to $y$. 
Let $\ell=d(1,y)$ be the projective distance to the identity matrix and let the sequence $\left(\,\left(\begin{smallmatrix} a_i^{-1}&0\cr 0&\bar a_i^T \end{smallmatrix}\right)\,\right)$ represent $y$.  So, writing $r_i:=\rk (a_i-1)$
for each $i$, we can assume that 
$\rk (a_i-\lambda1)\geq r_i-2$ for each $\lambda\neq0$.
It follows that 
$$\min_{\lambda\neq0}(\rk( a_i-\lambda^{-1}1)+\rk(a_i -\bar\lambda1))\geq 2r_i-4 .$$ Hence $n_i^{-1}r_i\to_\omega \ell$. 

%and the sequence $(a_i)$ represents an element of 
%$L=\prod_{i\to_\omega\infty}(\PSL_{n_i}(F_i),d)$ with distance $\ell$ from $1$. 

For each $i$ there is a matrix $x_i\in \GL_{n_i}(F_i)$ such that $x_i^{-1}a_ix_i$ has the form
$\big(\begin{smallmatrix} 1_{n_i-r_i} &v_i\cr0&w_i\end{smallmatrix}\big)$; here $w_i$ is invertible and does not have $1$ as an eigenvalue.  Let $w_i'$ be the result of multiplying $w_i$ by a diagonal matrix
with one diagonal entry $(\det w_i)^{-1}$ and all other diagonal entries $1$. Thus $w_i'\in\SL_{r_i}(F_i)$ and $\dim\rk(w_i'-1_{r_i})\leq1$.
By Lemma \ref{spaces} and Lemma \ref{splitSLn}, the metric
ultraproduct of the groups $\SL_{r_i}(F_i)$ with the ordinary rank metric is a geodesic space.  In this space the image $w$ of the sequence $(w_i')$ has length $1$.  Let
$(q_i(t))$ represent a geodesic from $1$ to $w$
in this space.  Thus for all $t_1,t_2\in [0,1]$ we have $\frac1{r_i}\rk(q_i(t_1)^{-1}q_i(t_2)-1)\to|t_1-t_2|$.  

For each $t\in[0,\ell]$ we now define 
$$q_i'(t)=x_i\begin{pmatrix} 1_{n_i-r_i}&0\cr 0& q_i(\ell^{-1} t)\end{pmatrix}x_i^{-1}\quad\hbox{and}\quad
p_i(t)=\begin{pmatrix} q_i'(t)^{-1}&0\cr 0&\overline {q_i'(t)^T}\end{pmatrix},$$ 
and let $(p_i(t))$ represent $p(t)$ in $G$.  Thus $p(0)=1$, $p(\ell)=y$.  Moreover
for $t_1,t_2\in[0,\ell]$ we have
$$\begin{array}{rl} %d(p(t_i),p(t_2))=
&\displaystyle\frac1{2n_i}\min_{\lambda\neq0}\left(\rk\begin{pmatrix} (q_i'(t_1)^{-1}q_i'(t_2))^{-1}&0\cr 0& \overline{q_i'(t_1)^{-1}q_i'(t_2))}^T\end{pmatrix}-\lambda1\right)\cr \cr
\leq&\displaystyle\frac1{2n_i}\big( \rk ((q_i'(t_1)^{-1}q_i'(t_2))^{-1}-1)+\rk (\overline{q_i'(t_1)^{-1}q_i'(t_2))}^T-1)\big) \cr\cr
=&\displaystyle\frac{r_i}{n_i}\,\frac1{r_i}\rk((q_i(\lambda^{-1}t_1)^{-1}q_i(\lambda^{-1}t_2))^{-1}-1)\cr
 \to &|t_1-t_2|\quad\hbox{as }i\to_\omega\infty.
\end{array}$$ Therefore $d(p(t_1),p(t_2))\leq |t_1-t_2|$ for all $t_1,t_2$ and since $d(p(0),p(\ell))=\ell$ it follows that $p$ is a geodesic from $1$ to $y$, as required.
\end{proof}
\section{Metric ultraproducts of symplectic, unitary and orthogonal groups}

For the properties of symplectic, unitary and orthogonal groups that we shall use, we refer the reader to \cite[Chapter 2]{KL} and \cite[Chapter 1 and \S 11.3]{carter}.

\begin{lem}\label{symplecticproduct}  Let $G=\Sp_{2n}(F)$, and regard $G$ as the subgroup of $\SL_{2n}(F)$ of block matrices  
$\left( \begin{smallmatrix} a& b\cr c& d\end{smallmatrix} \right)$ with $ad^T-bc^T=1_n$ and $ab^T$, $cd^T$ symmetric. Then $G=U^TUU^TH$ where
$$U= \left\{ \begin{pmatrix} 1& k\cr 0& 1\end{pmatrix}\mid k\hbox {  symmetric} \right\},\quad
U^T= \left\{ \begin{pmatrix} 1& 0\cr k& 1\end{pmatrix}\mid k\hbox {  symmetric} \right\}$$
and
$$H= \left\{ \begin{pmatrix} a^{-1}& 0\cr 0& a^T\end{pmatrix}\mid a\in\GL_{n}(F) \right\}.$$
\end{lem}

\begin{proof}  Clearly $U$, $U^T$, and $H$ are subgroups of $G$.
Let $g=\left( \begin{smallmatrix} a& b\cr c& d\end{smallmatrix} \right)\in G$. 
There are matrices $y_1,y_2\in\GL_n(F)$ with $y_1ay_2=\left( \begin{smallmatrix} 1_r&0 \cr 0& 0\end{smallmatrix} \right)$ where $r=\rk(a)$. Thus we can find $x_1,x_2\in H$ such that
$x_1gx_2$ has the form $\left( \begin{smallmatrix} a_1& b_1\cr c_1& d_1\end{smallmatrix} \right)$ with $a_1=\left( \begin{smallmatrix} 1_r& 0\cr 0& 0\end{smallmatrix} \right)$.  Write
$b_1=\left( \begin{smallmatrix} P& Q\cr R& S\end{smallmatrix} \right)$ with $P$ an $r\times r$ matrix.
Since $b_1a_1^T$ is symmetric we have $R=0$, and hence $S$ is invertible since the matrix $( a_1\; b_1)$ must have rank $n$.
Let 
$x_3=\left( \begin{smallmatrix} 1& 0\cr k& 1\end{smallmatrix} \right)$, where $k=\left( \begin{smallmatrix} 0& 0\cr 0& 1_{n-r}\end{smallmatrix} \right)$;
then $x_1gx_2x_3$ has the form $\left( \begin{smallmatrix} a_2& b_2\cr c_2& d_2\end{smallmatrix} \right)$ where $a_2$ is the invertible matrix $\left( \begin{smallmatrix} 1& Q\cr 0& S\end{smallmatrix} \right)$.  Adjusting the matrix $x_1\in H$ to effect further row operations in the matrix $(a_2\;b_2)$,
we can now assume that $x_1gx_2x_3$ has the form $\left( \begin{smallmatrix} 1& b_2\cr c_2& d_2\end{smallmatrix} \right)$.  Now $b_2$, $c_2$ are
symmetric and $d_2^T-b_2c_2^T=1$; and defining $x_4=\left( \begin{smallmatrix} 1& -b_2\cr 0& 1\end{smallmatrix} \right)$ we find that $\left( \begin{smallmatrix} 1& b_2\cr c_2& d_2\end{smallmatrix} \right) x_4=\left( \begin{smallmatrix} 1& 0\cr c_2& 1\end{smallmatrix} \right)$.  It follows that 
$G=HU^TUU^TH$.  However $H$ normalizes $U$ and $U^T$, and so $G=U^TUU^TH$, as required.  \end{proof}

The conclusion of Theorem \ref{main} for metric ultraproducts of symplectic groups (over not necessarily finite fields) follows immediately from Lemmas \ref{unipotentpath} and \ref{symplecticproduct}.

%For the properties of symplectic, unitary and orthogonal groups that we shall use, we refer the reader to \cite[Chapter 2]{KL} and \cite[Chapter 1 and \S 11.3]{carter}.

Modifications to the proof of Lemma \ref{symplecticproduct}, together with Lemma \ref{unipotentpath}, give the conclusion of the main theorem for metric ultraproducts of some other classical groups.

\begin{lem}\label{unitaryproduct}  Let $G=\Omega^+_{2n}(F)$  or $G=\SU_{2n}(F)$ with $n\geq3$.
In the latter case, there is an involution $\bar {\phantom n}$ on $F;$ in the former case we define
 $\bar {\phantom n}$ to be the identity map on $F$.  We regard $G$ as the subgroup of block matrices in $\SL_{2n}(F)$ of the form
$\big( \begin{smallmatrix} a& b\cr c& d\end{smallmatrix} \big)$ with 
$$a\bar d^T+b\bar c^T=1_n,\quad a\bar b^T+b\bar a^T=0\quad\hbox{and} \quad c\bar d^T+d\bar c^T=0.\eqno(\ast\ast)$$ Then $G= VV^TV(H\cap G)$ where
$$V= \left\{\left( \begin{matrix} 1& k\cr 0& 1\end{matrix} \right)\mid
\bar k^T=-k\right\},\quad V^T=\left\{\left( \begin{matrix} 1& 0\cr k& 1\end{matrix} \right)\mid \bar k^T=-k\right\}$$ and
$$H= \left\{ \left( \begin{matrix} a^{-1}& 0\cr 0& \bar a^T\end{matrix} \right)\mid a\in \GL_n(F)\right\}. $$   
\end{lem}

\begin{proof} Clearly $V$, $V^T$ are subgroups of $G$.  We have $H\leq G$ if and only if $\det a$ is fixed by the automorphism \ $\bar{}$ \ for each $a\in\GL_n(F)$. In particular this holds if \ $\bar{}$ \ is the identity map on $F$, and in this case the conclusion takes the form $G= VV^TVH$.

Let $g=\left( \begin{smallmatrix} a& b\cr c& d\end{smallmatrix} \right)\in G$. Arguing as in Lemma \ref{symplecticproduct} we find $x_1,x_2\in H$ such that
$x_1gx_2$ has the form $\big( \begin{smallmatrix} a_1& b_1\cr c_1& d_1\end{smallmatrix} \big)$ with $a_1=\left( \begin{smallmatrix} 1_r& 0\cr 0& 0\end{smallmatrix} \right)$ and
$b_1=\left( \begin{smallmatrix} P& Q\cr 0& S\end{smallmatrix} \right)$, where $P$ is an $r\times r$ matrix and $S$ is invertible.  Changing $x_1\in H$ to effect further row operations on $(a_1\;b_1)$, we can assume that $S=1_{n-r}$ and $Q=0$.  

We claim that there is an $(n-r)\times(n-r)$ matrix $k_1$ satisfying 
$k_1+\overline{k_1}^T=0$ and with rank $n-r$ except possibly in the case when $n-r$ is odd, $\mathrm {char}\;F\neq2$, and \ $\bar{}\;=\mathrm{id}_F$; and that in this latter case there   is an $(n-r)\times(n-r)$ matrix $k_1$ satisfying 
$k_1+\overline{k_1}^T=0$ and with rank at least $n-r-1$.  We begin with a matrix of block diagonal form 
 with $\lfloor\frac12(n-r)\rfloor$ diagonal entries $\big(\begin{smallmatrix} 0&1\cr -1&0\end{smallmatrix}\big)$.   If $n-r$ is odd, we must add a final diagonal entry $\lambda$ satisfying $\bar\lambda=-\lambda$. If 
 $\mathrm{char}\; F=2$ we may take $\lambda=1$. If  $\mathrm{char}\; F\neq2$ and \ $\bar{}\;\neq\mathrm{id}_F$ then $F$ is obtained from the fixed field of \ $\bar{}$ \ by adjoining a square root, and we may take $\lambda$ to be this square root.  Otherwise, we take $\lambda=0$.  Our claim follows.
 
 %We note that $F$ has an element $\lambda$ satisfying $\bar\lambda=-\lambda$ if either $\mathrm{char} F=2$ or \ $\bar{}\;\neq\mathrm{id}$ and $\mathrm{char} F\neq2$.  In the former case we may take $\lambda=1$, while in the latter, $F$ is obtained from the fixed field of the automorphism $\bar{}$ by adjoining a square root, and we can take $\lambda$ to be this square root.   It follows that there is an $(n-r)\times(n-r)$ matrix $k_1$ satisfying $k_1+\overline{k_1}^T=0$ and with rank $n-r$ if $\mathrm{char}\, F=2$ or if $\bar{}$ is not the identity, and of rank $n-r+1$ if $\mathrm{char}\, F\neq2$.  (For instance, one can take $k_1$ of block diagonal form with $\lfloor\frac12(n-r)\rfloor$ diagonal entries $\big(\begin{smallmatrix} 0&1\cr -1&0\end{smallmatrix}\big)$ and if $n-r$ is odd a diagonal entry $1$ if $\mathrm{char}\,F=2$ and $0$ otherwise.)  
Let $k$ be the $n\times n$ block matrix $\big(\begin{smallmatrix}0&0\cr 0&k_1\end{smallmatrix}\big)$ and $x_3$ the element $\big(\begin{smallmatrix} 1&0\cr k&1\end{smallmatrix}\big)$ of $V^T$.  
 An easy calculation shows that 
$x_1gx_2x_3$ has the form $\big( \begin{smallmatrix} a_2& b_2\cr c_2& d_2\end{smallmatrix} \big)$ where $\rk(a_2)\geq \rk (k_1)$.  

Suppose that $a_2$ is not invertible; so $n$ is odd, ${\rm char}\,F\neq2$ and \ $\bar{}$ \ is the identity.
In particular, $H\leq G$. 
Adjusting $x_1$ as above, we can assume that the top half of $x_1gx_2x_3$ has form
$\big(\begin{smallmatrix} 1_{n-1}&0&P&0\cr 0&0&0&1\end{smallmatrix}\big)$, with $P$ an $(n-1)\times (n-1)$ matrix and each $0$ a row or column vector. From $(\ast\ast)$ we have $P+\bar P ^T=0$, and so $y\in V^T$ where $$y=\begin{pmatrix} 1_{n-1}&0&0&0 \cr 0&1&0&0 \cr -P&0&1_{n-1}&0\cr 0&0&0&1\end{pmatrix}.$$ Therefore $z:=x_1gx_2x_3y\in G$, and $\det z=1$.  Now $z$ has the form
$$\begin{pmatrix} 1_{n-1}&0&0&0\cr 0&0&0&1\cr c_{11}&c_{12}&d_{11}&d_{12}\cr
c_{21}&c_{22}&d_{21}&d_{22}\end{pmatrix},$$ 
and row expansion yields $$\det z=-\det\begin{pmatrix}
c_{12}&d_{11}\cr c_{22}&d_{21}\end{pmatrix}.$$  
However from $(\ast\ast)$ we have
$$\begin{pmatrix} 1_{n-1}&0\cr 0&0\end{pmatrix}\begin{pmatrix}d_{11}&d_{21}\cr d_{12}&d_{22}\end{pmatrix}+\begin{pmatrix} 0&0\cr 0&1\end{pmatrix}\begin{pmatrix}c_{11}&c_{21}\cr c_{12}&c_{22}\end{pmatrix}=1_n,$$ and hence $d_{11}=1$, $d_{21}=0$, $c_{12}=0$, $c_{22}=1$.
Therefore $\det z=-\det\big(\begin{smallmatrix} 0& 1_{n-1}\cr 1&0\end{smallmatrix}\big), $ and this is equal to $-1$ because $n$ is odd.  Since $\mathrm{char}\, F\neq2$ this is a contradiction.

It follows that $a_2$ is invertible. Replacing $x_1$ by $\big(\begin{smallmatrix} a_2^{-1}&0\cr 0&\overline{a_2}^T\end{smallmatrix}\big)x_1$ we can assume that $a_2=1$.  The equations $(\ast\ast)$ show that $b_2+\bar b_2^T=0$, and setting $x_4=\left( \begin{smallmatrix} 1& -b_2\cr 0& 1\end{smallmatrix} \right)\in V$ we find that $x_1gx_2x_3x_4\in V^T$.  
Hence $g\in HV^TVV^TH$ and $g\in V^TVV^TH$ as $H$ normalizes $V$, $V^T$. Our conclusion now follows.
\end{proof}

The above lemma shows that the conclusion of Theorem \ref{main} holds for metric ultraproducts of
(not necessarily finite) groups of type $\PSU_{2n}(F)$ and of type ${\rm P}\Omega^+_{2n}(F)$.  To complete the proof of Theorem \ref{main} it suffices to prove that all remaining metric ultraproducts of finite simple groups that arise are isomorphic to such ultraproducts.

\begin{prop}\label{exhaust} \begin{enumerate} \item[\rm(a)] Every metric ultraproduct $\prod_{i \to \omega}(\PSU_{n_i}(F_i),d)$ with $n_i\to_\omega\infty$ is isomorphic to such an ultraproduct with all $n_i$ even.  
\item[\rm(b)] Every metric ultraproduct of finite simple orthogonal groups of $\omega$-unbounded degree is isomorphic to an ultraproduct of groups ${\rm P}\Omega^+_{2m_i}(F_i)$ with $m_i\to_\omega\infty$. 
\end{enumerate}\end{prop}

For the proof we need the following lemmas.

\begin{lem}\label{enough}  Let $V=L\oplus S$ be an orthogonal decomposition of a space $V$ with a non-singular sesquilinear or quadratic form and let $g\colon V\to V$ be a form-preserving isomorphism.  Then there is a form-preserving isomorphism $h\colon L\to L$ with $\rk(g|_L-h)\leq 3\dim S$.  \end{lem}

\begin{proof} We recall that every subspace $W$ of $V$ contains a subspace $D$ with $V=D\oplus D^\perp$ and $\dim (W/D)\leq \dim (V/W)$. If $\dim W\leq\frac12\dim V$ this is clear; if not, then $W$ has a non-zero subspace $C$ on which the form is non-singular, and the assertion follows by induction on consideration of $C^\perp$ and its subspace $C^\perp\cap W$. 

Let $M=g^{-1}(L)$; thus $\dim(V/(L\cap M))\leq 2\dim S$, and there is a subspace $U\leq L\cap M$ with $V=U\oplus U^\perp$ and with $\dim(V/U)\leq 4\dim S$.  The map
$g|_U$ is a form-preserving map to a subspace of $L$; since the form restricted to $L$ is non-singular, by Witt's Lemma $g|_{U}$ can be extended to a form-preserving isomorphism $h\colon L\to L$. Thus $h|_U=g|_U$ and so $\rk(g|_L-h)\leq \dim (L/U)\leq 3\dim S$.
\end{proof}

\begin{lem}  For each $i\in\NN$ let $V_i$ be a vector space of finite dimension $n_i$ having a non-singular sesquilinear or quadratic form $f_i$, let $V_i$ be the orthogonal direct sum of subspaces $L_i$, $S_i$
and let $G_i'$, $H_i'$ be the derived groups of the groups $G_i$, $H_i$ of form-preserving maps
on $V_i$, $L_i$.  Suppose that the sequence $(\dim S_i)$ is bounded.  Then the maps $H_i\to G_i$ induce an isomorphism 
of the metric ultraproducts $H$, $G$ of the quotients $H'_i/\rZ(H'_i)$ and the quotients $G'_i/\rZ(G'_i)$.  \end{lem}

\begin{proof}  Suppose that $\dim S_i\leq \kappa$ for each $i$. We may regard $H_i$ as the subgroup of $G_i$ fixing $S_i$ pointwise.  We may also assume that all quotients $H_i/\rZ(H_i)$, $G_i/\rZ(G_i)$ are simple; thus the inclusion maps $H'_i\to G'_i$ induce maps $H'_i/\rZ(H'_i)\to G'_i/\rZ(G'_i)$ and induce a map $H\to G$ that is evidently injective.  To prove that it is surjective, it will suffice to prove that it contains the commutator of any pair $g^{(1)}, g^{(2)}$ of elements of the simple group $G$.

Let  the sequences $(g_i^{(1)})$, $(g_i^{(2)})$ be preimages in the product of the groups
$(G_i')$ of $g^{(1)}, g^{(2)}$.  Using Lemma \ref{enough}, for each
$i$ we construct $h_i^{(1)}, h_i^{(2)} \in \GL(L_i)$ preserving the form, with $$\rk(g^{(k)}|_{L_i}-h_i^{(k)})\leq 3\kappa\quad\hbox{for  } k=1,2.$$  It follows
that for all non-zero $\lambda_i^{(1)}, \lambda_i^{(2)}$ in the field we have $$|\rk(g_i^{(k)}-\lambda_i^{(k)} \id) -\rk(h_i^{(k)}-\lambda_i^{(k)}\id)|\leq 6\kappa\quad\hbox{for  } k=1,2,$$
and hence for $k=1,2$ the two sequences $(h_i^{(k)})$, $(g_i^{(k)})$ map to the same element of $G$. The result now follows easily.  
\end{proof}   

\begin{proof}[Proof of Proposition {\rm \ref{exhaust}.}]  The simple unitary and orthogonal groups are the quotients modulo the centre of derived groups of isometry groups of spaces with a non-singular unitary form and spaces with a non-singular quadratic form.  
Every finite-dimensional space with a non-singular unitary form has an orthonormal basis (see for example \cite[Proposition 2.3.1]{KL}); and so it can be written as an orthogonal direct sum $L\oplus S$ with $\dim L$ even and $\dim S\leq1$.  

Let $V$ be a finite space with a non-singular quadratic form $Q$ and associated bilinear form $\phi$.  Then $V$ has a basis of one of the following types:
\begin{enumerate}
\item[(i)] $\{e_1,\dots, e_m,f_1\dots, f_m\}$ with $Q(e_i)=Q(f_i)=0 $ and $\phi(e_i,f_j)=\delta_{ij}$ for all $i,j$;
\item[(ii)]  $\{e_1,\dots, e_m,f_1\dots, f_m,x,y\}$ with $Q(e_i)=Q(f_i)=0 $ and $\phi(e_i,f_j)=\delta_{ij}$
for all $i,j$ and with the subspace spanned by $x,y$ an orthogonal direct summand;
\item[(iii)]  $\{e_1,\dots, e_m,f_1\dots, f_m,x\}$ with $Q(e_i)=Q(f_i)=0 $ and $\phi(e_i,f_j)=\delta_{ij}$
for all $i,j$ and with the subspace spanned by $x$ an orthogonal direct summand. \end{enumerate}
(See \cite[Proposition 2.5.3]{KL}.)  Take $L$ to be the space spanned by the elements $e_i,f_j$ and
$S$ to be the space spanned by the remaining basis elements.  In each case we obtain an orthogonal decomposition $V=L\oplus S$ with $\dim S\leq2$ and with the form induced on $L$ having isometry group with derived group $\Omega^+_{2m}$.  Therefore the above lemmas can be applied and  
Proposition \ref{exhaust} follows. \end{proof}
 
 \section*{Acknowledgments}

The research leading to these results has received funding from the European Research Council under the European Union's Seventh Framework Programme -- ERC Starting Grant $n^{\circ}\ 277728$. This paper was written while the second author was Leibniz Professor at the University of Leipzig, and he thanks this university for its hospitality.

\end{document}